\documentclass[12pt]{amsart}
\usepackage[T1]{fontenc}
\usepackage[utf8]{inputenc}
\usepackage{color,amssymb,amsmath,mathrsfs, enumerate,esint}
\usepackage{a4wide}
\usepackage{hyperref}
\usepackage{pgf,tikz,pgfplots}
\pgfplotsset{compat=1.15}
\usepackage{mathrsfs}
\usetikzlibrary{arrows}
\theoremstyle{plain}
\newtheorem{Theorem}{Theorem}[section]
\newtheorem{Lemma}[Theorem]{Lemma}

\newtheorem{Observation}[Theorem]{Observation}
\newtheorem{Corollary}[Theorem]{Corollary}

\theoremstyle{definition}

 \DeclareMathOperator{\dee}{d}

\DeclareMathOperator{\er}{\mathbb{R}}

 \DeclareMathOperator*{\di}{dist}
\author{Tomáš Roskovec}
\author{Filip Soudský}

\address{T.~G.~Roskovec: Faculty of Economics, University of South Bohemia, Studentsk\' a 13, \v Cesk\' e Bud\v ejovice, Czech Republic; 
}
\email{troskovec@jcu.cz}

\address{F.~Soudsk\' y: Faculty of Science, Humanities and Education, Technical University of Liberec, Studentská 1402/2, Liberec, Czech Republic; 
}
\email{filip.soudsky@tul.cz}

\title{An elementary proof of Acerbi Fusco minimizer existence theorem}
\begin{document}
\maketitle
\begin{abstract}
    The weak lower semicontinuity of the functional 
    $$
    F(u)=\int_{\Omega}f(x,u,\nabla u)\dee x
    $$
    is a classical topic that was studied thoroughly. It was shown that if the function $f$ is continuous and convex in the last variable, the functional is sequentially weakly lower semicontinuous on $W^{1,p}(\Omega)$. However, the known proofs use advanced instruments of real and functional analysis. Our aim here is to present proof that can be easily understood by students familiar only with the elementary measure theory. 
\end{abstract}

\section{Introduction and the main result}
Many optimization problems in physics may be formulated in terms of minimization of the functional in the form of
\begin{equation}\label{eq:Fu}\tag{M}
F(v)=\int_\Omega f(x,v,\nabla v)\dee x,    
\end{equation}
over some set $\mathcal{W}$ of admissible functions. This problem is a very classical one and its roots can be traced to the 17th century. The systematic study was launched later by Lagrange, du Boys-Reymond, and others (see for instance \cite{Lagr, duB}). They formulated the Euler-Lagrange equations as necessary conditions to hold for the minimizer. However, the existence of the minimizer has to be proved first (otherwise even solving Euler-Lagrange equations does not guarantee to find the solution). As many examples show, the existence of the minimizer cannot be expected in general. Therefore we have to add some additional assumptions on the function $f$ in \eqref{eq:Fu} to assure the existence of a minimizer. The existence problem, however, is typically approached with advanced methods including the knowledge of certain parts of functional analysis, measure theory, and function spaces. We present an alternative proof of the existence of a minimizer under the convexity condition of $f$ with a given boundary data. Unlike the classical ones, our proof is using only elementary methods of measure theory and real analysis and its direction is easy to follow. Our calculation is done in detail and we intentionally avoid shortening it with advanced tools.

Now we specify the problem. Let us consider a domain $\Omega\subset \er^n$ with a Lipschitz boundary and let $f\in \mathcal{C}(\overline{\Omega}\times \er\times\er^n)$. Note that the condition of continuity of $f$ can be relaxed and one may consider the Caratheodory function instead (see \cite{DAC}), however, this is done easily using the Scorza-Dragoni theorem \ref{thmSD}, thus we restrict ourselves to continuous functions. Let $u_0\in W^{1,p}(\Omega,\er^d)$, we denote the set of all functions in given Sobolev space with the same boundary data as $u_0$ by
\begin{equation}\label{Mnoz}
 \mathcal{W}:=u_0+W^{1,p}_0(\Omega,\er^d).   
\end{equation}
Let $p\in(1,\infty)$, by \textit{p-growth condition} of $f$ we mean, if there exists $1\leq q< p$ such that
\begin{equation}\label{PG}\tag{p-G}
\begin{aligned}
  (\exists c_0>0)(\exists c_1\in\er,c_2\in L^1(\Omega))(\forall (x,u,\xi)\in\overline{\Omega}\times\er^d\times\er^z)\\
f(x,u,\xi)\geq c_0|\xi|^p+c_1|u|^q+c_2(x).   
\end{aligned}
\end{equation}
The proof can be slightly simplified by considering $c_1$ and $c_2$ positive constants, as some authors do. Also, $d$ is the dimension of the image of functions in \eqref{Mnoz}, which can be $d=1$ but also higher for vector-valued functions, we leave $d$ general in statements, but we focus on the case $d=1$. Note that the case of $p=1$ follows immediately from Lemma \ref{LSCP}. However, since the Sobolev space $W^{1,1}$ is not reflexive, even the sequential weak lower semi-continuity does not guarantee the existence of the minimizer. The case $p=\infty$ is also covered in some books (for instance see \cite{DAC}), however, this would require a different technique. We consider the convexity property of $f$ in the last variable
\begin{equation}\label{C}\tag{conv}
\xi\mapsto f(x,u,\xi) \quad  \textup{is convex for all}\quad  (x,u)\in\Omega\times \er^d. 
\end{equation}
In the following text we consider only functional $F$ and set $\mathcal{W}$ satisfying 
\begin{equation}\label{Fin}\tag{fin}
    (\exists u\in\mathcal{W})\quad  F(u)<\infty\quad \wedge\quad (\forall u\in \mathcal{W})\quad F(u)>-\infty.
\end{equation}

Our problem is to prove that the minimizer of \eqref{eq:Fu} on $\mathcal{W}$ exists under conditions \eqref{C}, \eqref{Fin}, and \eqref{PG}. Note that the existence follows immediately from Theorems \ref{MR1} and \ref{ODHAD}. Our contribution is the new student-friendly proof of Theorem \ref{MR1}. The reader is encouraged to compare our proof with the classical proofs in  \cite{ACF, DAC, Marc} or \cite{Giu}, in these proofs useful tools are both approximations by the functions with controlled growth and regularity or replacing and extending the functions on the irregular parts of the domain, also the maximal operator is used. The shorter proof was done later by Ka{\l}amajska \cite{Agnieszka} with help of the theory of Young measures. For the sake of completeness and convenience of the reader, we also include proof of Theorem \ref{ODHAD} although this one is similar to the known ones.

\begin{Theorem}[Weak sequential lower semicontinuity]\label{MR1}
Let $\Omega\subset \er^n$ be a domain with a Lipchitz boundary and let $u_0\in W^{1,p}(\Omega,\er^d)$. Let $f\in\mathcal{C}(\overline{\Omega}\times \er^d\times \er^{nd})$ be a function such that \eqref{C} holds.
Then the functional $F$ given by \eqref{eq:Fu} is sequentially weakly lower semicontinuous.
\end{Theorem}

\begin{Theorem}[Coercivity]\label{ODHAD}
Let $\Omega\subset \er^n$ be a domain with a Lipchitz boundary and let $u_0\in W^{1,p}(\Omega,\er^d)$. Let $f\in\mathcal{C}(\overline{\Omega}\times \er^d\times\er^{nd})$ be a function satisfying \eqref{PG}. Let $u_k\in \mathcal{W}$ given by \eqref{Mnoz} be a minimizing sequence of functional $F$ given by \eqref{eq:Fu}. Moreover, let condition \eqref{Fin} hold. Then $u_k$ is bounded in $W^{1,p}(\Omega,\er^d)$. 
\end{Theorem}

\begin{Corollary}[Existence of minimizer]\label{tutosumato}
Let $\Omega\subset \er^n$ be a domain with a Lipchitz boundary.  Let $f\in\mathcal{C}(\overline{\Omega}\times\er^d\times\er^{nd})$ be a function satisfying \eqref{C} and \eqref{PG}. Let $F$ be the functional given by \eqref{eq:Fu} satisfying \eqref{Fin}. Then $F$ has a minimizer in $\mathcal{W}$ defined by \eqref{Mnoz}.
\end{Corollary}

The paper is structured into Section \ref{Sekce 2}, where we prepare some observations and lemmata and prove Theorem \ref{ODHAD}, Section \ref{Sekce 3}, where we prove the main result, and Appendix \ref{Sekce 4}, where we recall some well-known theorems for convenience of less experienced readers. 

\subsection{Development and connection to literature}

As for the development in the literature, the result we reprove is of Acerbi and Fusco in 1984 \cite{ACF}. The methods they used include maximal operators, an extension of functions, and a lower semicontinuous envelope, the paper also contains the proof even for quasi-convex vector valued $f$. The result succeeded the previous results by Tonelli \cite{Ton}, where the Theorem \ref{MR1} was proved with much stricter conditions. Later the conditions were relaxed in the regularity of $f$, for development of this see \cite{Serr, EkTe} or \cite{MaSb}. For vector-valued $u$, the assumption of convexity is no longer necessary and can be relaxed to various weaker properties. Originally, the result was given by Morrey \cite{Morr} and generalized by Meyers \cite{Mey} under stronger regularity requiring only quasi-convexity. More results were published by Ball, Liu, Dacorogna, Marcellini, and others with finer properties under more relaxed assumptions such as polyconvexity or rank-one convexity. As the definitions of particular generalizations of convexity differ, we are led to study differences and inclusions between these classes, the topic was addressed by \v{S}ver\'{a}k \cite{Sver} and Alibert and Dacorogna \cite{ADac} or recent results by Sil \cite{Sil} and Grabovsky \cite{Grab}. The ongoing trend is to ensure an improvement of convexity type by studying additional properties \cite{up1}. Note that even though \eqref{PG} conditions can not be left out, they can be relaxed into finer scales than powers \cite{VeZe}. An important tool is controlling $f$ by the convex envelope, see \cite{BaKiKr} for studying its properties and dependence on growth conditions. Let us emphasise that the variation problems without convexity are also approached, see \cite{Eke1, Eke}.

Let us briefly overview some milestones and interesting papers in the field surrendering or following the Acerbi-Fusco theorem \cite{ACF}. In 1982, an important result about the regularity of the minimizer was done by Giaquinta and Giusti \cite{GiaGiu}, for a wide survey on this topic we recommend the paper by Mingione \cite{min}. Approximation methods by Marcellini \cite{Marc} improve the Acerbi-Fusco theorem in 1985. Significant relaxation result was presented in 1997 by Fonseca and Mal\'y \cite{FoMa}. A version of convexity based on curl called {{{A}}}-convexity initially considered by Dacorogna \cite{Dac2} is shown to be the optimal variant of definition for the lower semicontinuity property by Fonseca and M\"uller in 1999 \cite{FoMu}. The Young measure theory is used by Ka{\l}amajska in 1997 to shorten the proof of the Acerbi-Fusco result. Also by the theory of Young measure, Kristensen in 1999 provided new approximation results in \cite{Krist} and in 2015 he redefine growth conditions for gradient Young functions and characterise the lower semicontinuity in this setting \cite{Krist2}. The case of the lower semicontinuity among $W^{1,1}$ and BV functions was studied in this setting by Kristensen and Rindler in 2010 \cite{KriRi}. Recent result by Prinari covers the lower semicontinuity and approximation properties for $L^\infty$ functionals \cite{Pri}. The splendid result by Bourdin, Francfort, and Gilles put the variational approach to address Griffith fracture models \cite{BFM} in 2007. The proper list of references would be overwhelming thus we had to omit a lot.

\section{Some auxiliary results and observations}\label{Sekce 2}
In this and the following text we use measure theory, but even though the formulation of lemmata and observations are more general, we apply them only for the Lebesgue measure in spaces $\mathbb{R}$ and $\mathbb{R}^n$.

\begin{Observation}\label{LP1}
Let $(\Omega,\mu)$ be a finite measure space and let $f\in L^0(\Omega)$ and $g\in L^1(\Omega)$ be functions for which $g(x)\leq f(x)$ holds for a.e. $x\in\Omega$. Let $\Omega_k\subset \Omega$ be sets such that $|\Omega\setminus\Omega_k|\to 0$. Then
$$
\lim_k\left(\int_{\Omega_k} f\dee \mu\right)=\int_{\Omega} f\dee \mu.
$$
\end{Observation}

Note, that \textit{partition of a set $M$} means the pair-wise disjoint family $M_m$ such that 
$$
\bigcup_m M_m=M.
$$
Given a partition $\mathcal{P}$ we denote the \textit{norm of the partition} by
$$
\nu(\mathcal{P}):=\max_{j\in m}\left\{\textup{diam}(P_j)\right\}.
$$
\begin{Lemma}\label{ApiM}
Let $\Omega\subset \er^n$ be a bounded domain, $u\in L^1(\Omega)$ and let $\mathcal{P}=(P_j)_{j=1}^m$ be a partition of $\Omega$. Define
$$
u_{\mathcal{P}}:=\sum_{j=1}^m\chi_{P_j}\fint_{P_j} u.
$$
Then
$$
\lim_{\nu(\mathcal{P})\to 0}|\{|u_{\mathcal{P}}-u|>\varepsilon\}|=0\quad(\forall \varepsilon>0).
$$
\end{Lemma}

\begin{proof}
    Let us first show that the statement is valid for continuous functions. Pick $\varepsilon,\gamma>0$. Given a positive $\eta$ let us denote
    $$
    \Omega_\eta:=\{x\in\Omega: \di(x,\Omega^c)>\eta\}.
    $$
    Now we choose $\eta$ so small that
    $$
    |\Omega\setminus \Omega_{2\eta}|<\gamma.
    $$
    Since $u$ is uniformly continuous on $\Omega_\eta$, there exists $\delta>0$ such that 
    $$
    |x-y|<\delta\Rightarrow |u(x)-u(y)|<\varepsilon \quad \forall x,y\in\Omega_\eta. 
    $$
    Now choose a partition $\mathcal{P}$ with
    \begin{equation}\label{mez}
     \nu({\mathcal{P}})<\min\{\eta,\delta\}.   
    \end{equation}
    For all $x\in\Omega_{2\eta}$ we have that 
    $$
    |u(x)-u_{\mathcal{P}}(x)|<\varepsilon.
    $$
    Hence
    $$
    |\{|u-u_{\mathcal{P}}|>\varepsilon\}|<\gamma
    $$
    as soon as \eqref{mez} holds. The convergence in measure for the continuous function is proved.

 Given a $\alpha,\varepsilon>0$ and $u\in L^1(\Omega)$ we first choose $v$ continuous such that
$$
\|u-v\|_{L^1}\leq \varepsilon \alpha,
$$
then by the triangle inequality
    $$
    \begin{aligned}
    \left|\{|u-u_{\mathcal{P}}|>3\varepsilon\}\right|&\leq |\{|u-v|>\varepsilon\}|+ |\{|v-v_{\mathcal{P}}|>\varepsilon\}|+|\{|u_{\mathcal{P}}-v_{\mathcal{P}}|>\varepsilon\}|    \end{aligned}\\
    =:I+II+III.
    $$
Note that $u\mapsto u_{\mathcal{P}}$ is non-expansive mapping in $L^1$, so we estimate
$$
\|u_{\mathcal{P}}-v_{\mathcal{P}}\|_{L^1}\leq \varepsilon\alpha.
$$
Hence, by Chebyschev inequality \ref{thmChe} we have $I+III<2\alpha$. Since $v$ is continuous it is enough to choose the norm of the partition small to obtain $II<\alpha$.
\end{proof}
\begin{proof}[Proof of Theorem \ref{ODHAD}]
Let $\tilde{u}_k\in \mathcal{W}$ be a minimizing sequence of $F$. Denote
\begin{equation}\label{eq:oddel}
u_k:=\tilde{u}_k-u_0.
\end{equation}
First, we prove that the sequence is bounded in $W^{1,p}(\Omega)$. Let us suppose the contrary that $u_k$ is unbounded. Without loss of generality (otherwise pass to a proper subsequence) suppose 
$$
0<\|u_k\|_{W^{1,p}}\to \infty,
$$
the Friedrich inequality thm. \ref{thmFr} implies
$$
\|\nabla u_k\|_p\to\infty.
$$
This implication does not have to be true in general, but thanks to $u_k\equiv 0$ on the boundary it has to hold.
Now, by \eqref{PG} we estimate
$$
\begin{aligned}
F(\tilde{u}_k)&\geq c_0\|\nabla \tilde{u}_k\|_p^p-|c_1|\|\tilde{u}_k\|_q^q+\|c_2\|_{1}\\
&\stackrel{\eqref{eq:oddel}}{\geq} c_0(\|\nabla u_k\|_p^p-\|\nabla u_0\|_p^p)-|c_1|(\|u_k\|_{q}^q+\|u_0\|_q^q)+\|c_2\|_{1}\\
&\stackrel{\text{Thm. \ref{thmFr}}}{\geq} c_0\|\nabla u_k\|_p^p-C(\Omega)c_1\|\nabla u_k\|_q^q+C(u_0,c_1,c_2, \Omega)\\
&\stackrel{\text{H\"older in.}}{\geq} c_0\|\nabla u_k\|_p^p-C(c_1,\Omega)\|\nabla u_k\|_p^{q}+C(u_0,c_1,c_2,\Omega)\\
&=\|\nabla u_k\|_p^p\left(c_0-C(c_1,\Omega)\|\nabla u_k\|^{q-p}-\|\nabla u_k\|_{p}^{-p}C(u_0,c_1,c_2,\Omega)\right).
\end{aligned}
$$
As $p>q$, the second and the third term in the bracket tend to zero and $c_0>0$, so the remaining expression diverges to infinity as $k$ tends to infinity. But this is a contradiction to the fact that $u_k$ is a minimizing sequence.

\end{proof}

\section{Proof of Theorem \ref{MR1} and Corollary \ref{tutosumato}}\label{Sekce 3}
Let $u\in L^0(\Omega,\er^d)$ and $v\in L^p(\Omega, \er^z)$. Given such functions let us consider
\begin{equation}\label{funkcJ}
J_{\Omega}(u,v):=\int_\Omega f(x,u(x),v(x))\dee x.    
\end{equation}
Let us denote the norm topology in $L^p$ by $\mathcal{L}^p$ and the topology of weak convergence by $\mathcal{L}^p_w$. We also denote the convergence in the Lebesgue measure by $\stackrel{\dee x}{\to}$. Let us define the topology
$$
\tau=\mathcal{L}^0\times\mathcal{L}^p_w
$$
for $(u_k,v_k)\in L^0(\Omega,\er^d)\times L^p(\Omega,\er^z)$ by
$$
(u_k,v_k)\stackrel{\tau}{\to}(u,v) \quad \textup{if }\left(u_k\stackrel{\dee x}{\to}u\wedge v_k\stackrel{L^p}{\rightharpoonup} v\right).
$$

\begin{Lemma}\label{LSCP}
Let $\Omega\subset \er^n$ be a bounded domain, $p\in[1,\infty)$, $f\in\mathcal{C}(\overline{\Omega}\times \er^d\times \er^z)$ be function satisfying \eqref{C} and let  $b\in L^1(\Omega)$ be a function such that 
\begin{equation}\label{star}
f(x,u,v)\geq b(x)\quad \forall (x,u,v)\in\Omega\times\er^d\times \er^z.
\end{equation}
Then the functional $J_{\Omega}$ defined in \eqref{funkcJ} is sequentially lower semicontinuous with respect to topology $\tau$. 
\end{Lemma}

\begin{Lemma}\label{Jcon}
Let $\Omega\subset \er^n$ be a bounded domain, $p\in(1,\infty)$ and let $f\in\mathcal{C}(\overline{\Omega}\times\er^d\times\er^z)$ be a function satisfying \eqref{PG} and \eqref{C}. Then $J_{\Omega}$ defined in \eqref{funkcJ} is sequentially lower semicontinuous with respect to topology $\mathcal{L}^q\times \mathcal{L}^p_{w}$.    
\end{Lemma}

\begin{proof}[Proof of Lemma \ref{LSCP}]
Choose $\varepsilon>0$. Let $(u_k,v_k)$ be a sequence convergent to $(u,v)$ in the topology $\tau$. Since $u_k$ converges to $u$ in measure, by the Riesz thm. \ref{thmRi} one may assume that $u_k(x)\to u(x)$ for a.e. $x\in\Omega$. Hence by the Yegorov thm. \ref{thmYEG} and the Luzin thm. \ref{thmLu} for every $l$ natural there exists a compact set $\Omega_l\subset\Omega$ such that 
$$
\left(|\Omega\setminus\Omega_l|\leq \frac{1}{l}\right)\wedge \left(u_k\rightrightarrows u \quad\textup{on }\Omega_l\right)\wedge \left(u\in \mathcal{C}(\Omega_l)\right).
$$
Define 
$$
S:=\sup_k\sup_{x\in\Omega_l}|u_k(x)|.
$$
We may assume that $S<\infty$. Note that the function $(x,u)\mapsto f(x,u,\xi)$ is uniformly continuous on $\Omega_l\times [-S,S]\times\overline{B(0,2j)}$ for arbitrary $j$ natural. Now, let us denote by $\eta_j$ the strictly positive number such that for all $x,\overline{x}\in\Omega_l, u,\overline{u}\in[-S,S]$ 
\begin{equation}\label{DUL}
|(\overline{x},\overline{u})-(x,u)|<6\eta_j
\end{equation}
implies
$$
|f(\overline{x},\overline{u},\xi)-f(x,u,\xi)|<\varepsilon\quad (\forall |\xi|\leq 2j).
$$
Without loss of generality one may suppose that $\eta_j$ is decreasing and
\begin{equation}\label{RK}
|u(x)-u_k(x)|\leq \eta_k\quad(\forall x \in\overline{\Omega_l}).   
\end{equation}
Moreover, since the function $u$ is uniformly continuous, by the triangle inequality and double usage of \eqref{RK}, for each $j$ there exists positive   $\delta_j<\eta_j$ ($\delta_j$ decreasing) such that for all $x,y\in\overline{\Omega}$, one has 
$$
|x-y|<\delta_j\quad \textup{implies} \quad |(x,u_k(x))-(y,u_k(y))|<3\eta_j\quad(\forall k\geq j).  
$$
Use the previous estimates to observe
\begin{equation}\label{sit2}
|x-y|<\delta_j\quad \textup{implies} \quad |(x,u_k(x))-(y,u(y))|<6\eta_j\quad(\forall k\geq j).  
\end{equation}
For $j,k$ naturals denote
\begin{equation}\label{defEjkGj}
E_{k,j}:=\{|v_k|\leq j\}\cap\Omega_l,\quad G_j:=\{|v|\leq j\}\cap\Omega_l.
\end{equation}
Note that, using the Chebychev inequality thm. \ref{thmChe} and boundedness of $v_k$ in $L^p$ we obtain that there exists $C>0$ independent of $j,k$ such that 
\begin{equation}\label{mir}
|\Omega_l\setminus G_j|\leq \frac{C}{j^p}\quad \text{ and }\quad |\Omega_l\setminus E_{k,j}|\leq \frac{C}{j^p}.
\end{equation}

By uniform continuity of $(x,u,\xi)\mapsto f(x,u,\xi)$ on $\Omega_l\times[-S,S]\times \overline{B(0,2j)}$, for every natural $j$ there exists $\gamma_j$ such that  
\begin{equation}\label{PANTAU}
|\xi-\overline{\xi}|<\gamma_j\Rightarrow |f(x,u,\xi)-f(x,u,\overline{\xi})|<\varepsilon\quad(\forall(x,u)\in\overline{\Omega_l}\times [-S,S]\quad \textup{and}\quad \forall \xi,\overline{\xi}\in \overline{B(0,2j)}
\end{equation}

By Lemma \ref{ApiM} we obtain that for each $j$ natural there exists a finite partition of $G_j$, consider such a partition $\mathcal{P}=(K^j_m)_{m=1}^{M_j}$ satisfying
\begin{enumerate}[\upshape(i)]
    \item $\nu(\mathcal{P})<\delta_j$.
    \item
    $
    \sum_m\left|\{x\in K^j_m: |v-\fint_{K^j_m}v|>\gamma_j\}\right|<1/j.$
\end{enumerate}
By the weak convergence of $v_k$ to $v$ in $L^p$ one has
$$
\lim_k \fint_{K^j_m}(v-v_k)=0.
$$
Therefore, one may assume without the loss of generality
\begin{equation}\label{konvprum}
 \left|\fint_{K^j_m}(v-v_k)\right|\leq \gamma_k \quad (\forall k\geq j)   
\end{equation}
(if this doesn't hold one can pass to the appropriate sub-sequence).
\noindent
Set 
\begin{equation}\label{BytostQ}
    Q^j_m:= \left\{x\in K^j_m: |v(x)-\fint_{K^j_m}v|\leq \gamma_j\right\}
\end{equation}
and denote 
\begin{equation}\label{BytostQ naporcovana}
Q_j:=\bigcup_m Q^j_m.
\end{equation}
Note that by \eqref{defEjkGj} and the property (ii) of the parts $K^j_m$ one has $ |G_j\setminus Q_j|\leq \frac{1}{j}.$
Pick $a_j>j$ such that 
\begin{equation}\label{ajaj}
\left|\fint_{K^j_m} v_j\chi_{E^c_{j, a_j}}\dee x\right|< \gamma_j.
\end{equation}
Moreover, let $Q_{m,t}^j$ be a partition of $Q^j_m$ such that
$$
\textup{diam}(Q_{m,t}^j)< \delta_{a_j},
$$
and let us choose arbitrary $x^j_{m,t}\in Q_{m,t}^j$, creating $2\delta_{a_j}$ net. For the following estimate denote
\begin{equation}\label{znackaT}\begin{aligned}
\xi_{j,m}:=&\fint_{K^j_m} v\dee x, & \xi^\spadesuit_{j,m}:=&\fint_{K^j_m} v_j\dee x, \\
\xi^\heartsuit_{j,m}:=&\fint_{K^j_m} v_j\chi_{E_{j,a_j}}\dee x, &\tilde{v}_j:=&v_j\chi_{E_{j,a_j}}, \\
T:=&\max_{(x,u)\in\overline{\Omega}_l\times[-S,S]}|f(x,u,0)|.&&
\end{aligned}\end{equation}
Note that \eqref{ajaj} estimate the difference of functions above as
\begin{equation}\label{ajajaj}
\left|\xi^\spadesuit_{j,m}-\xi^\heartsuit_{j,m}\right|=\left|\fint_{K^j_m} v_j\chi_{E^c_{j, a_j}}\dee x\right|< \gamma_j.
\end{equation}
Let us estimate 
\begin{eqnarray*}
\int_{Q_j}f(x,u,v)\dee x&\stackrel{\eqref{sit2}, \eqref{DUL}}{\leq}& \sum_m \int_{Q^j_m}f(x^j_m,u_j(x^j_m),v)\dee x+|\Omega_l|\varepsilon\\
&\stackrel{\eqref{BytostQ}, \eqref{PANTAU}}{\leq}& \sum_m \int_{Q^j_m}f(x^j_m,u_j(x^j_m),\xi_{j,m})\dee x+2|\Omega_l|\varepsilon\\
&\stackrel{\eqref{konvprum}, \eqref{PANTAU}}{\leq}& \sum_m \int_{Q^j_m}f(x^j_m,u_j(x^j_m),\xi^\spadesuit_{j,m})\dee x+3|\Omega_l|\varepsilon\\
&\stackrel{\eqref{sit2}, \eqref{DUL}}{\leq}& \sum_{m}\sum_t\int_{Q_{m,t}^j}f(x^j_{m,t},u_j(x^j_{m,t}), \xi^\spadesuit_{j,m})\dee x+4|\Omega_l|\varepsilon\\
&\stackrel{\eqref{ajajaj}, \eqref{PANTAU}}{\leq}& \sum_{m}\sum_t\int_{Q_{m,t}^j}f(x^j_{m,t},u_j(x^j_{m,t}),\xi^\heartsuit_{j,m}) \dee x+5|\Omega_l|\varepsilon\\
&\stackrel{\text{Thm }\ref{thmJe}}{\leq}&\sum_{m}\sum_t \int_{Q_{m,t}^j}f(x^j_{m,t},u_j(x^j_{m,t}), \tilde{v}_j)\dee x+5|\Omega_l|\varepsilon.
\end{eqnarray*}
Note that the usage of the Jensen inequality in the last estimate is the crucial step, where convexity \eqref{C} comes to play. Note that in this step, we may question if more general versions such as quasi-convexity can be used. Following the fact that $x^j_{m,t}$ is a 2$\delta_{a_j}$ net and the inequalities \eqref{DUL},\eqref{sit2} and \eqref{RK}, we estimate
\begin{eqnarray*}
&\leq& \sum_m\sum_t \int_{Q_{m,t}^j}f(x,u_j, \tilde{v}_j)\dee x+6|\Omega_l|\varepsilon\\
&\stackrel{\eqref{BytostQ naporcovana},\text{part. of } Q_m^j}{=}&\int_{Q_j}f(x,u_j,\tilde{v}_j)\dee x+6|\Omega_l|\varepsilon.
\end{eqnarray*}
To estimate we split th domain into $Q_j\cap E_{j, a_j}$ where $\tilde{v}_j=v_j$ and $Q_j\cap E_{j, a_j}^c$, where we estimate $|f|\leq T$. In the last estimate, we add to the right-hand side integral of non-negative $f-b$ over the set $\Omega_l\setminus(E_{j,a_j}\cap Q_j)$, so we estimate
\begin{eqnarray*}
&\stackrel{\eqref{znackaT}}{\leq}& \int_{E_{j,a_j}\cap Q_j}f(x,u_j,v_j)\dee x+T|E_{j,a_j}^c|+6|\Omega_l|\varepsilon\\
&\stackrel{\eqref{mir}}{\leq}&  \int_{\Omega_l}f(x,u_j,v_j)\dee x+\frac{TC}{a_j^p}+6|\Omega_l|\varepsilon-\int_{\Omega_l\setminus(E_{j,a_j}\cap Q_j)}b(x)\dee x.
\end{eqnarray*}
Passing to the limit with $\varepsilon\to 0_+$ yields
$$
J_{Q_j}(u,v)\leq J_{\Omega_l}(u_j,v_j)+\frac{TC}{a_j^p}-\int_{\Omega_l\setminus(E_{j,a_j}\cap Q_j)}b(x)\dee x.
$$
 Using Observation \ref{LP1} and \eqref{star}, passing to limit for $j\to\infty$ we obtain
$$
J_{\Omega}(u,v)\leq \limsup_{j}J_{\Omega_l}(u_j,v_j),
$$
which is equivalent to sequential lower semicontinuity of $J_{\Omega}$ on the topological space $(L^0(\Omega)\times L^p(\Omega),\tau)$.
\end{proof}

\begin{proof}[Proof of the Lemma \ref{Jcon}]
Given a function satisfying \eqref{PG}, define the auxiliary function
$$
\tilde{f}(x,u,\xi):=f(x,u,\xi)-c_1|u|^q
$$
and functional 
$$
\tilde{J}(u,v):=\int_{\Omega}\tilde{f}(x,u,v)\dee x.
$$
By the Lemma \ref{LSCP} the functional $\tilde{J}$ is sequentially lower semicontinuous with respect to topology $\tau$. Hence it is sequentially lower semicontinuous with respect to $\mathcal{L}^q\times \mathcal{L}^p_w$ and thus
$$
    J(u,v)=\tilde{J}(u,v)+c_1\int_{\Omega}|u|^q\dee x\leq \liminf_{k}\tilde{J}(u_k,v_k)+c_1\lim_k\int_{\Omega}|u_k|^q\dee x=\liminf_{k}J(u_k,v_k).
$$
\end{proof}
\begin{proof}[Proof of Theorem \ref{MR1}]
Let  $u_k\rightharpoonup u$ in $W^{1,p}$. Using the Rellich-Kondrachev theorem \ref{thmRK} without loss of generality, we may suppose that $u_k\to u$ strongly in $L^p$ while $\nabla u_k\stackrel{L^p}{\rightharpoonup} \nabla u$. Hence by the preceding Lemma \ref{Jcon}, we obtain that 
$$
F(u)=J(u,\nabla u)\leq \liminf_k J(u_k,\nabla u_k)=\liminf_k F(u_k).
$$
\end{proof}
\begin{proof}[Proof of Corollary \ref{tutosumato}]
Let $f\in\mathcal{C}(\Omega\times\er\times\er^n)$ be a function satisfying \eqref{PG} and \eqref{C}. Let $u_k$ be a minimizing sequence. By Theorem \ref{ODHAD} we have that $u_k$ is bounded in $W^{1,p}(\Omega)$. Alaoglu theorem implies that the closed ball in $W^{1,p}(\Omega)$ is compact, therefore we may assume that $u_k$ is weakly convergent to $u\in W^{1,p}(\Omega)$. Let us emphasise that by the Rellich-Kondrachev theorem \ref{thmRK} we also may assume that $u_k\to u$ in $L^q(\Omega)$. We get
$$
(u_k,\nabla u_k)\to (u,\nabla u) \quad \textup{in the topology}\quad \mathcal{L}^q\times \mathcal{L}^p_w.
$$
Using the Theorem \ref{MR1} we obtain that
$$
F(u)\leq \liminf_{k\to\infty}F(u_k)=\inf\{F(u)|u\in\mathcal{W}\}.
$$
Hence $u$ is the minimizer.
\end{proof}

\section{Appendix}\label{Sekce 4}
The following results are standard knowledge, the proofs can be found in \cite{LuMa}, \cite{DAC}, \cite{Leo} etc.
\begin{Theorem}[Scorza-Dragoni]\label{thmSD}
Let $\varepsilon>0$ and $f:\Omega\times \er\times \er^n\to \er$ be measurable in the first variable and continuous in the second and third one. Then there exists compact set $K\subset\Omega$ such that $|\Omega\setminus K|<\varepsilon$ and $f|_{K\times \er\times\er^n}$ is continuous.
\end{Theorem}

\begin{Theorem}[Friedrichs inequality]\label{thmFr}
Let $\Omega\subset \er^n$ be a domain with Lipchitz boundary, $p\in\left<1,\infty\right>$. There exists constant $C_\Omega$ such that
$$
\|u\|_p\leq C_{\Omega}\|\nabla u\|_p
$$
holds for all $u\in W^{1,p}_0(\Omega)$.
\end{Theorem}

\begin{Theorem}[Jensen inequality]\label{thmJe}
Let $N\subset \er^n$ be a convex set. Let $\mu$ be a measure with a finite total variation on $N$. Let $\varphi:N\to \er$ be a convex function. Let $E\subset N$ be a measurable set.  Then the following inequality holds
$$
\varphi\left(\fint_E x\dee \mu(x)\right)\leq \fint_E\varphi(x)\dee \mu(x).
$$
\end{Theorem}

\begin{Theorem}[Riesz]\label{thmRi}
Let $(\Omega,\mu)$ be a space with measure and let $u_k$ be a sequence of measurable functions such that 
$$
u_k\stackrel{\mu}{\to} u\quad u_k\textup{ converge in measure to } u.
$$
then there exists an increasing sequence of indices $l_k$ such that 
$$
u_{l_k}(x)\to u(x)\quad \textup{for } \mu-\textup{ a.e. }x\in\Omega.
$$
\end{Theorem}
\begin{Theorem}[Rellich-Kondrachev]\label{thmRK}
Let $p\in(1,\infty)$ and let $1\leq q<p^*$ where
$$
p^*:=\begin{cases}
\frac{np}{n-p}\quad &\textup{for }p<n,\\
\infty &\textup{otherwise.}
\end{cases}
$$
Let $\Omega\subset \er^n$ be a bounded domain with Lipchitz boundary. And let $u_k\in W^{1,p}(\Omega)$ be a bounded sequence in $W^{1,p}(\Omega)$, then there exists an increasing sequence of indices $l_k$ and $u\in L^q(\Omega)$ such that 
$$
\|u-u_{l_k}\|_{L^q(\Omega)}\to 0.
$$
\end{Theorem}

\begin{Theorem}[Yegorov]\label{thmYEG}
Let $\Omega\subset \er^n$ measurable, let $u_k$ be a sequence of measurable functions such that 
$$
u_k(x)\to u(x)\quad \textup{for }\mu-\textup{a.e. } x\in\Omega.
$$
Then there exist the sequence of sets $N_j$ and the strictly increasing sequence of indices $k_j$ such that
\begin{enumerate}[\upshape(Y1)]
    \item 
$
N_j\subset N_{j+1}, 
$
\item 
$
\mu(\Omega\setminus N_j)\leq \frac{1}{j},
$
\item
$
|u-u_{k_j}|<\frac{1}{j}\quad \textup{on } N_j.
$
\end{enumerate}

\end{Theorem}

\begin{Theorem}[Luzin] \label{thmLu}
Let $\mu$ be a complete Radon measure on locally compact space and measurable function $u$ finite $\mu$-almost everywhere. Then for open $K$ and $\varepsilon>0$ there exist open $G\subset K$ such that $u$ is continuous on $K\setminus G$ and $\mu(G)<\varepsilon.$ 
\end{Theorem}

\begin{Theorem}[Chebysev inequality] \label{thmChe}
Let $\Omega\subset \er^n$ measurable, $u\in L^0(\Omega)$, $p\in\left<1,\infty\right)$ and let $t\in(0,\infty)$ then one has
$$
t^p\left|\left\{|u|>t\right\}\right|\leq \int_{\Omega}|u(x)|^p\dee x.
$$
\end{Theorem}

\section*{Acknowledgments}

Both named authors were supported by grant number GJ20-19018Y of the Grant Agency of the Czech Republic.

\bibliographystyle{plain}

\end{document}